\documentclass[12pt,reqno]{amsart}

\usepackage{amssymb}

\usepackage[all]{xy}

\numberwithin{equation}{section}

\newtheorem{theorem}{Theorem}[section]
\newtheorem{proposition}[theorem]{Proposition}
\newtheorem{lemma}[theorem]{Lemma}

\theoremstyle{definition}

\newtheorem{remark}[theorem]{Remark}

\begin{document}

\baselineskip=15pt

\title[Holomorphic foliations with no transversely projective
structure]{Holomorphic foliations with no transversely projective structure}

\author[I. Biswas]{Indranil Biswas}

\address{Department of Mathematics, Shiv Nadar University, NH91, Tehsil Dadri,
Greater Noida, Uttar Pradesh 201314, India}

\email{indranil.biswas@snu.edu.in, indranil29@gmail.com}

\author[S. Dumitrescu]{Sorin Dumitrescu}

\address{Universit\'e C\^ote d'Azur, CNRS, LJAD, France}

\email{dumitres@unice.fr}



\date{}

\begin{abstract}
We prove that on the product of two elliptic curves a generic nonsingular turbulent holomorphic foliation does 
not admit any transversely holomorphic projective structure.
\end{abstract}

\maketitle

%\tableofcontents

\section{Introduction}

In \cite{Gh}, Ghys proved a classification result for nonsingular codimension one holomorphic foliations 
on compact complex tori. He proved that such foliations are either linear (e.g. they are defined as 
the kernel of a global nonzero holomorphic one-form on the torus) or {\it turbulent}.

We recall that a turbulent foliation on a compact complex torus ${\mathbb C}^n / \Lambda$ is constructed using 
a holomorphic fibration $\pi \,:\, {\mathbb C}^n \,\longrightarrow\, C$, over an elliptic curve $C$ and it is 
defined by a closed meromorphic one-form $\eta \,=\, (\pi^*\omega) + \beta$, where $\omega$ is a meromorphic 
one-form on $C$ and $\beta$ is a holomorphic one-form on ${\mathbb C}^n / \Lambda$ that does not vanish on the
fibers of $\pi$. These foliations are nonsingular (see \cite{Gh}).

A more general notion of a singular turbulent foliation was defined in \cite{Br2}, where Brunella 
classified all (possibly singular) codimension one foliations on compact complex tori. A general study of 
holomorphic foliations on compact homogeneous K\"ahler manifolds was carried out in \cite{LoP}.

The dynamics of the above mentioned turbulent foliation was described in \cite{Gh, Br2}: the inverse image 
through $\pi$ of the polar divisor of the meromorphic one-form $\beta$ is a finite union of compact leafs and 
all other leafs are noncompact and accumulate on every compact leaf.
It should be mentioned that turbulent foliations, being defined by global closed meromorphic one-forms, 
admit a {\it singular} transversely complex projective structure in the sense of \cite{LPT}. Also Brunella 
classified in \cite{Br1} non-singular holomorphic foliations on compact complex surfaces. A consequence of his 
study is that all those foliations admit a singular transversely complex projective structure.

This article deals with the question whether nonsingular turbulent foliations do admit nonsingular 
transversely projective structures. Notice that if the holomorphic one-form $\beta$ is zero (or vanishes on the 
fibers of $\pi$), the turbulent foliation degenerate to a fibration. All fibrations over a Riemann surface 
admit a nonsingular transversely projective structure which is a consequence of the uniformization theorem for 
Riemann surfaces.

We study turbulent foliations on the product of two elliptic curves. Our first result is that each turbulent 
foliation admits at most one nonsingular transversely complex projective structure (see Proposition 
\ref{prop1}).

The main result proved here is that for a polar part of the above meromorphic one-form $\omega$ of degree $d 
\,\geq\, 8$, generic turbulent foliations do not admit any nonsingular transversely complex projective structure 
(see Theorem \ref{thm1}).

To the best of our knowledge these are the first known examples of nonsingular codimension one holomorphic 
foliations on a complex projective manifold admitting no transversely complex projective structures.

The organization of the paper is as follows. Section \ref{se2} describes the details of the construction of 
turbulent foliations on the product of two elliptic curves. Section \ref{se3} introduces the classical notion 
of complex projective structure on a Riemann surface and presents the Ehresmann's equivalent description given 
by a flat ${\mathbb P}^1$-bundle with a transverse holomorphic section. Section \ref{se 4} gives the foliated 
version of the complex projective structures and also their description using foliated flat 
${\mathbb P}^1$-bundles. The last Section \ref{se 5} provides the proofs of the main results: The uniqueness 
of the transversely projective structure (Proposition \ref{prop1}) and the non-existence of a transversely 
projective structure for a generic turbulent foliation (Theorem \ref{thm1}).

\section{A turbulent foliation}\label{se2}

Take a compact connected Riemann surface $C$ of genus one. Fix an integer $d\,\geq\, 2$.
The group $S_d$ of permutations of $\{1,\, \cdots,\, d\}$ acts on $C^d$ by permuting the entries
in the Cartesian product. Let
$$
\text{Sym}^d(C)\,\, :=\,\, C^d/S_d
$$
be the symmetric product. The complement of the big diagonal in $\text{Sym}^d(C)$ will be denoted by
$\text{Sym}^d_0(C)$. So an element $\{x_1,\, \cdots,\, x_d\}\,\in \text{Sym}^d(C)$ lies in
$\text{Sym}^d_0(C)$ if and only if $\{x_1,\, \cdots,\, x_d\}$ are all distinct. For any
$\underline{x}\,=\, \{x_1,\, \cdots,\, x_d\}\,\in \text{Sym}^d_0(C)$, we have the reduced effective divisor
$\sum_{i=1}^d x_i$ on $C$ of degree $d$. For notational convenience, this
divisor will also be denoted by $\underline{x}$.

Take two points $\underline{x}\,=\, \{x_1,\, \cdots,\, x_d\},\,\, \underline{y}\,=\, \{y_1,\, \cdots,\, y_d\}\, 
\in\, \text{Sym}^d_0(C)$. Fixing an element $x_0\,\in\, C$ as the identity element, we can make $C$ into a
complex abelian Lie group. Note that the condition that the two elements $\sum_{i=1}^d x_i$ and $\sum_{i=1}^d y_i$
of $C$ coincide does not depend on the choice of the point $x_0$.

\begin{lemma}\label{lem1}
Take $\underline{x},\,\, \underline{y}\,\in\, {\rm Sym}^d_0(C)$ as above such that
$\{x_1,\, \cdots,\, x_d\}\cap\{y_1,\, \cdots,\, y_d\}\, =\, \emptyset$. If there is a
meromorphic function on $C$ whose zero divisor is $\underline{y}$ and the pole divisor is
$\underline{x}$, then the two elements $\sum_{i=1}^d x_i$ and $\sum_{i=1}^d y_i$ of $C$ coincide.
\end{lemma}

\begin{proof}
Fix a point $x_0\, \in\, C$ as the identity element, and consider $C$ as a group. Then $C$ is identified with 
the Jacobian $J^0(C)$ by the map that sends any $z\, \in\, C$ to the line bundle ${\mathcal O}_C(z-x_0)$. In 
this identification, the element ${\mathcal O}_C(\sum_{i=1}^d (x_i-y_i))\, \in\, J^0(C)$ is mapped to 
$\sum_{i=1}^d (x_i-y_i)\,\in\, C$. Consequently, if there is a meromorphic function on $C$ whose zero divisor 
is $\underline{y}$ and the pole divisor is $\underline{x}$, then $\sum_{i=1}^d (x_i-y_i)\,=\, x_0\,\in\, C$.
\end{proof}

The following lemma is a converse of Lemma \ref{lem1}.

\begin{lemma}\label{lem2}
Take two points $\underline{x}\,=\, \{x_1,\, \cdots,\, x_d\},\,\, \underline{y}\,=\, \{y_1,\, \cdots,\, y_d\}\, 
\in\, {\rm Sym}^d_0(C)$ such that
\begin{itemize}
\item $\{x_1,\, \cdots,\, x_d\}\cap\{y_1,\, \cdots,\, y_d\}\, =\, \emptyset$, and

\item the two elements $\sum_{i=1}^d x_i$ and $\sum_{i=1}^d y_i$ of $C$ coincide.
\end{itemize}
Then there is a meromorphic function on $C$ whose zero divisor is $\underline{y}$ and the pole divisor is
$\underline{x}$.

Any two such meromorphic functions differ by multiplication with a nonzero complex number.
\end{lemma}

\begin{proof}
As in the proof of Lemma \ref{lem1}, fix a point $x_0\, \in\, C$ as the identity element. From the
second condition we know that the element ${\mathcal O}_C(\sum_{i=1}^d (x_i-y_i))\, \in\, J^0(C)$ is mapped
to $\sum_{i=1}^d (x_i-y_i)\,=\, x_0\,\in\, C$ by the identification in the proof of Lemma \ref{lem1}.
Now $x_0$ corresponds
to the identity element ${\mathcal O}_C\,\in\, J^0(C)$. Therefore, an isomorphism between
${\mathcal O}_C(\sum_{i=1}^d (x_i-y_i))$ and ${\mathcal O}_X$ takes the section of ${\mathcal O}_X$, given
by a nonzero constant function on $C$, to a meromorphic function on $C$ whose zero divisor is
$\underline{y}$ and the pole divisor is $\underline{x}$.

The second statement of the lemma is obvious.
\end{proof}

The canonical line bundle of $C$, which is a holomorphically trivializable line bundle, will be denoted by $K_C$.
Take two points
\begin{equation}\label{e1}
\underline{x}\,=\, \{x_1,\, \cdots,\, x_d\},\,\,\,\, \underline{y}\,=\, \{y_1,\, \cdots,\, y_d\}\, 
\in\, \text{Sym}^d_0(C)
\end{equation}
such that
\begin{itemize}
\item $\{x_1,\, \cdots,\, x_d\}\cap\{y_1,\, \cdots,\, y_d\}\, =\, \emptyset$, and

\item the two elements $\sum_{i=1}^d x_i$ and $\sum_{i=1}^d y_i$ of $C$ coincide.
\end{itemize}
Since $K_C$ is holomorphically isomorphic to ${\mathcal O}_C$, from Lemma \ref{lem2} we know that there is a meromorphic
one-form $\omega$ on $C$ whose zero divisor is $\underline{y}$ and the pole divisor is $\underline{x}$. So we have
\begin{equation}\label{en2}
\omega\,\ \in\,\, H^0(C,\, K_C\otimes {\mathcal O}_C(\underline{x})),
\end{equation}
and this holomorphic section satisfies the conditions that $\omega(y_i)\,=\, 0$ for all $1\, \leq\, i\, \leq\, d$ and
$\omega(x_i)\,\not=\, 0$ for all $1\, \leq\, i\, \leq\, d$ (see \eqref{e1}).

Let $X$ be a compact connected Riemann surface of genus one; its canonical bundle will be denoted
by $K_X$. Let
\begin{equation}\label{e2}
\beta \,\, \in\,\, H^0(X,\, K_X)\setminus \{0\}
\end{equation}
be a nonzero holomorphic $1$--form.

Fix holomorphic trivialization of $K_C$. 

Let
\begin{equation}\label{e3}
\phi_C\,:\, C\times X\, \longrightarrow\, C \ \ \text{ and }\ \ \phi_X\,:\, C\times X\, \longrightarrow\, X
\end{equation}
be the natural projections. The holomorphic line bundle $\phi^*_C{\mathcal O}_C(-\underline{x})$ on $C\times X$ will
be denoted by $L$. Note that we have $\phi^*_C\omega\,\in\, H^0(C\times X,\, L^*)$ (see \eqref{en2}). So
$\phi^*_C\omega$ gives a homomorphism
\begin{equation}\label{k2}
L\,\, \longrightarrow\,\, {\mathcal O}_{C\times X}.
\end{equation}
Composing this homomorphism with the homomorphism
$$
{\mathcal O}_{C\times X} \,=\, \phi^*_C K_C \, \hookrightarrow\, (\phi^*_C K_C)\oplus (\phi^*_X K_X)
\,=\, \Omega^1_{C\times X},
$$
where the isomorphism ${\mathcal O}_{C\times X} \,=\, \phi^*_C K_C$ is given by the trivialization of $K_C$,
we get a homomorphism
\begin{equation}\label{k1}
\omega'\,\,:\,\, L\,\, \longrightarrow\,\, \Omega^1_{C\times X}.
\end{equation}
Composing the homomorphism in \eqref{k2} with the homomorphism
$$
{\mathcal O}_{C\times X}\, \xrightarrow{\,\,\, \phi^*_X \beta\,\,\,} \, \phi^*_X K_X
\, \hookrightarrow\, (\phi^*_X K_X)\oplus (\phi^*_C K_C) \,=\, \Omega^1_{C\times X},
$$
where $\beta$ is the $1$-form in \eqref{e2}, we get a homomorphism
\begin{equation}\label{k3}
\beta'\,\,:\,\, L\,\, \longrightarrow\, \, \Omega^1_{C\times X}.
\end{equation}
Now we have the homomorphism
\begin{equation}\label{e4}
\vartheta\,\,:=\,\, \omega'+ \beta'\,\,:\,\, L\,\, \longrightarrow\,\, \Omega^1_{C\times X} \oplus
\Omega^1_{C\times X},
\end{equation}
where $\omega'$ and $\beta'$ are constructed in \eqref{k1} and \eqref{k3} respectively.

We will show that the subsheaf $\vartheta(L)\, \subset\,
\Omega^1_{C\times X}$ in \eqref{e4} is a subbundle (equivalently,
$\Omega^1_{C\times X}/\vartheta(L)$ is locally free). To prove this, first note
that the homomorphism $\omega'$ in \eqref{k1} vanishes 
exactly on $$\phi^{-1}_C({\rm Div}(\omega))\,=\, \{y_1,\, \cdots,\, y_d\}\times X\,\subset\, C\times X.$$
On the other hand, the homomorphism $\beta'$ in \eqref{k3}
vanishes exactly on $$\{x_1,\, \cdots,\, x_d\}\times X\, \subset\, C\times D.$$
Since $\underline{x}$ and $\underline{y}$ are disjoint, we conclude that the subsheaf $\vartheta(L)\,
\subset\,\Omega^1_{C\times X}$ in \eqref{e4} is a subbundle.

Since $\vartheta(L)$ is a line subbundle of $\Omega^1_{C\times X}$,
\begin{equation}\label{e4a}
{\mathcal F}\,\, :=\,\, \text{kernel}(\vartheta(L)) \,\, \subset\,\, T(C\times X)
\end{equation}
is a nonsingular holomorphic $1$-dimensional foliation on $C\times X$.

\section{Projective structure on a Riemann surface}\label{se3}

\subsection{Definition of a projective structure}

Let $M$ be a connected Riemann surface. A holomorphic coordinate function on $M$ is a pair
$(U,\, f)$, where $U\, \subset\, M$ is an open subset and $$f\, :\, U\, \longrightarrow\,
{\mathbb C}{\mathbb P}^1\,=\, {\mathbb C}\cup \{\infty\}$$
is a holomorphic isomorphism between $U$ and $f(U)$. A holomorphic coordinate atlas on $M$ is a collection
of holomorphic coordinate functions $\{(U_i,\, f_i)\}_{i\in I}$ on $M$ such that $\bigcup_i U_i\,=\, M$.
A projective structure on $M$ is defined by a holomorphic coordinate atlas
$\{(U_i,\, f_i)\}_{i\in I}$ on $M$ such that for any $i,\, j\, \in\, I$, and any nonempty connected
component $V\, \subset\, U_i\bigcap U_j$, the map
$$
f_i\circ (f^{-1}_j\big\vert_{f_j(V)}) \,\, :\,\, f_j(V) \,\, \longrightarrow\,\,
f_i(V)\,\,\subset\, {\mathbb C}{\mathbb P}^1
$$
is the restriction of a M\"obius transformation; see \cite{Gu}, \cite{De}. Two such 
holomorphic coordinate atlases $\{(U_i,\, f_i)\}_{i\in I}$ and $\{(U_i,\, f_i)\}_{i\in J}$ on $M$
define the same projective structure if the holomorphic coordinate atlas on $M$ given
by their union $\{(U_i,\, f_i)\}_{i\in I\cup J}$ also satisfies the above condition.
A holomorphic coordinate function $(U,\, f)$ is said to be compatible with the projective structure
on $M$ defined by a holomorphic coordinate atlas $\{(U_i,\, f_i)\}_{i\in I}$ if the union
$\{(U_i,\, f_i)\}_{i\in I}\bigcup \{(U,\, f)\}$ defines a projective structure on $M$.

The holomorphic cotangent and tangent bundles of $M$ will be denoted by $K_M$ and $TM$ respectively.

There are projective structures on $M$, for example, the uniformization of $M$ produces a projective
structure on $M$. The space of all projective structures on $M$ is an affine space for the vector space
$H^0(M,\, K^{\otimes 2}_M)\,=\, H^0(M,\, ((TM)^{\otimes 2})^*)$ (see \cite{Gu}, \cite{Hu}). 
This affine space structure of the space of all projective structures on $M$ will be recalled below.

\subsection{A flat projective bundle}\label{se3.2}

For a holomorphic vector bundle $V$ on a Riemann surface $M$, and any nonnegative integer $i$,
the $i$-th order holomorphic jet bundle for $V$ will be denoted by $J^i(V)$. The fiber of
$J^i(V)$ over any point $x\, \in\, M$ is the space of all holomorphic sections of $V$ over the
$i$-th order infinitesimal neighborhood of $x$ (see \cite[Part 4, Ch.~6]{Gro}, \cite{De}, \cite{BR} for
the construction of the jet bundles).

The Lie bracket operation on the holomorphic vector fields on ${\mathbb C}{\mathbb P}^1$ produces
a Lie algebra structure on $H^0({\mathbb C}{\mathbb P}^1,\, T{\mathbb C}{\mathbb P}^1)$. The
resulting Lie algebra is isomorphic to $sl(2, {\mathbb C})\,=\, \text{Lie}({\rm SL}(2,
{\mathbb C}))\,=\, \text{Lie}({\rm PSL}(2, {\mathbb C}))$; an isomorphism between the two Lie
algebras is given by the standard action of ${\rm PSL}(2, {\mathbb C})$ on ${\mathbb C}{\mathbb P}^1$. Consider
the trivial vector bundle on ${\mathbb C}{\mathbb P}^1$
\begin{equation}\label{j3}
{\mathbb C}{\mathbb P}^1\times H^0({\mathbb C}{\mathbb P}^1,\, T{\mathbb C}{\mathbb P}^1)
\, \longrightarrow\, {\mathbb C}{\mathbb P}^1
\end{equation}
with fiber $H^0({\mathbb C}{\mathbb P}^1,\, T{\mathbb C}{\mathbb P}^1)$. Let
\begin{equation}\label{j2}
\beta\,\, :\,\, {\mathbb C}{\mathbb P}^1\times H^0({\mathbb C}{\mathbb P}^1,\, T{\mathbb C}{\mathbb P}^1)
\, \longrightarrow\, J^2(T{\mathbb C}{\mathbb P}^1)
\end{equation}
be the natural evaluation homomorphism. If a global section $w\, \in\,
H^0({\mathbb C}{\mathbb P}^1,\, T{\mathbb C}{\mathbb P}^1)$ vanishes of order three at
a point $x\, \in\, {\mathbb C}{\mathbb P}^1$, then $w$ actually vanishes identically; indeed, this
follows immediately from the fact that $\text{degree}(T{\mathbb C}{\mathbb P}^1)\,=\, 2$. Consequently, the homomorphism
$\beta$ in \eqref{j2} is fiberwise injective. This implies that $\beta$ in \eqref{j2} is an isomorphism because
$\dim H^0({\mathbb C}{\mathbb P}^1,\, T{\mathbb C}{\mathbb P}^1)\,=\, \text{rank}(J^2(T{\mathbb C}{\mathbb P}^1))$.
Using the isomorphism $\beta$ in \eqref{j2}, the Lie algebra structure on the
fibers of the trivial vector bundle in \eqref{j3} produces a Lie algebra structure on the fibers
of $J^2(T{\mathbb C}{\mathbb P}^1)$. The trivial connection on the vector bundle
in \eqref{j3} is the unique holomorphic connection on it. This connection induces a holomorphic
connection on $J^2(T{\mathbb C}{\mathbb P}^1)$ using $\beta$; let
\begin{equation}\label{e5}
\nabla^0
\end{equation}
be this induced holomorphic connection on $J^2(T{\mathbb C}{\mathbb P}^1)$. The connection $\nabla^0$
in \eqref{e5} is clearly compatible with the Lie algebra structure of the fibers of $J^2(T{\mathbb C}{\mathbb
P}^1)$, meaning the Lie bracket of two locally defined flat sections of $J^2(T{\mathbb C}
{\mathbb P}^1)$ is also flat.

The standard action of the M\"obius group $\text{PGL}(2,{\mathbb C})$ on
${\mathbb C}{\mathbb P}^1$ has a natural lift to $T{\mathbb C}{\mathbb P}^1$, and
hence $\text{PGL}(2,{\mathbb C})$ acts on both $J^2(T{\mathbb C}{\mathbb P}^1)$ and
$H^0({\mathbb C}{\mathbb P}^1,\, T{\mathbb C}{\mathbb P}^1)$. The connection $\nabla^0$ (see
\eqref{e5}) and the Lie algebra structure of the fibers of $J^2(T{\mathbb C}{\mathbb P}^1)$ are both 
$\text{PGL}(2,{\mathbb C})$--invariant. Equip
${\mathbb C}{\mathbb P}^1\times H^0({\mathbb C}{\mathbb P}^1,\, T{\mathbb C}{\mathbb P}^1)$ with the
diagonal action of $\text{PGL}(2,{\mathbb C})$. Then the homomorphism
$\beta$ in \eqref{j2} is evidently $\text{PGL}(2,{\mathbb C})$--equivariant.

Let $M$ be a connected Riemann surface equipped with a projective structure $\mathcal P$. 
Take any holomorphic coordinate function $(U,\, f)$ on $M$ which is compatible with the projective structure
$\mathcal P$. Using the differential, of the map $f$,
\begin{equation}\label{e6}
df\, :\, TU\, \longrightarrow\, f^*T{\mathbb C}{\mathbb P}^1,
\end{equation}
identify the jet bundle $J^2(TU)\,=\, J^2(TM)\big\vert_U$ with
$f^*J^2(T{\mathbb C}{\mathbb P}^1)\,=\, J^2(f^*T{\mathbb C}{\mathbb P}^1)$. Using this identification,
the Lie algebra structure on the fibers of $J^2(T{\mathbb C}{\mathbb P}^1)$ produces a Lie
algebra structure on the fibers of $J^2(TU)\,=\, J^2(TM)\big\vert_U$. This Lie algebra structure
on the fibers of $J^2(TM)\big\vert_U$ is actually independent of the choice of the coordinate function
$f$. Indeed, this follows immediately from the fact that the Lie algebra structure of the fibers of
$J^2(T{\mathbb C}{\mathbb P}^1)$ is preserved by the action of $\text{PGL}(2,{\mathbb C})$
on $J^2(T{\mathbb C}{\mathbb P}^1)$. Consequently,
we obtain a Lie algebra structure on the fibers of $J^2(TM)$. The Lie algebra structure
on every fiber of $J^2(TM)$ is isomorphic to $sl(2, {\mathbb C})$ because the Lie
algebra structure on every fiber of $J^2(T{\mathbb C}{\mathbb P}^1)$ is identified with
$sl(2, {\mathbb C})$. It should be clarified that the isomorphism between $sl(2, {\mathbb C})$ and a fiber of 
$J^2(TM)$ does depend on the choice of $f$.

Using the above identification $J^2(TM)\big\vert_U\,=\, f^*J^2(T{\mathbb C}{\mathbb P}^1)$
induced by $df$ in \eqref{e6}, the pulled back connection $f^*\nabla^0$ (see \eqref{e5}) produces a
holomorphic connection on $J^2(TM)\big\vert_U$. Again, this connection on $J^2(TM)\big\vert_U$
is independent of the choice of the coordinate function
$f$ because $\nabla^0$ is preserved by the action of $\text{PGL}(2,{\mathbb C})$ on
$J^2(T{\mathbb C}{\mathbb P}^1)$. Consequently, we obtain a holomorphic connection on $J^2(TM)$. Let
\begin{equation}\label{e7}
\nabla^{\mathcal P}_0
\end{equation}
be the holomorphic connection on $J^2(TM)$ constructed this way. Both $\nabla^{\mathcal P}_0$ and
the Lie algebra structure on the fibers of $J^2(TM)$ very much depend on the projective structure $\mathcal P$.

Let ${\mathbb P}(J^2(TM))\, \longrightarrow\, M$ be the projective bundle parametrizing the
lines in the fibers of $J^2(TM)$. The connection $\nabla^{\mathcal P}_0$ in \eqref{e7} produces a
holomorphic connection on this projective bundle. Recall that the fibers of $J^2(TM)$ are
Lie algebras isomorphic to $sl(2, {\mathbb C})$. Let
\begin{equation}\label{e10}
{\mathbb K}\,\, :\,\, J^2(TM)\otimes J^2(TM)\,\, \longrightarrow\,\, {\mathcal O}_M
\end{equation}
be the Killing form on the fibers of $J^2(TM)$. The locus of all $v\, \in\, {\mathbb P}(J^2(TM))$
such that ${\mathbb K}(v\otimes v)\,=\, 0$ is a projective bundle
\begin{equation}\label{e8}
\varphi\,\, :\,\, {\mathbf P}_{\mathcal P}\,\, \longrightarrow\,\, M
\end{equation}
of relative dimension one (the fibers are isomorphic to ${\mathbb C}{\mathbb P}^1$). The Lie
algebra structure on the fibers of $J^2(TM)$ is compatible with the connection $\nabla^{\mathcal P}_0$
(see \eqref{e7}), because the connection $\nabla^0$ in \eqref{e5} compatible with the Lie
algebra structure on the fibers of $J^2(T{\mathbb C}{\mathbb P}^1)$. Therefore, the
connection on ${\mathbb P}(J^2(TM))$ induced by $\nabla^{\mathcal P}_0$ actually preserves the
subbundle ${\mathbf P}_{\mathcal P}$ in \eqref{e8}. Let
\begin{equation}\label{e12}
\nabla^{\mathcal P}
\end{equation}
be the holomorphic connection on the fiber bundle ${\mathbf P}_{\mathcal P}$ induced by $\nabla^{\mathcal P}_0$.

The jet bundle $J^2(TM)$ fits in the short exact sequence of holomorphic vector bundles on $M$
\begin{equation}\label{e11}
0\, \longrightarrow\, K^{\otimes 2}_M\otimes TM\,=\, K_M \, \longrightarrow\, J^2(TM)
\, \longrightarrow\, J^1(TM) \, \longrightarrow\, 0.
\end{equation}
It can be shown that ${\mathbb K}(K_M\otimes K_M) \,=\, 0$ (see \eqref{e11}), where $\mathbb K$ is the
Killing form in \eqref{e10}. To see this, note that if $s$ and $s_1$ are two locally defined
holomorphic vector fields defined on a neighborhood of $x\, \in\, M$ such that $s$ vanishes
at $x$ of order two, and $s_1$ vanishes at $x$ of order $j$ (to clarify, $j$ may be
$0$, meaning $s_1(x)\, \not=\, 0$ is allowed), then the Lie bracket $[s,\, s_1]$ vanishes at $x$ of order
at least $j+1$. Now from the construction of the Lie algebra structure on the fibers of
$J^2(T{\mathbb C}{\mathbb P}^1)$ we conclude that the above property of $[s,\, s_1]$ ensures that
the subbundle $K_M$ in \eqref{e11} lies in the nilpotent locus of the fibers of $J^2(TM)$. Consequently,
we have ${\mathbb K}(K_M\otimes K_M) \,=\, 0$. This implies that $K_M$ in \eqref{e11} defines a
holomorphic section 
\begin{equation}\label{e9}
\sigma\,\,:\,\, M\,\, \longrightarrow\,\, {\mathbf P}_{\mathcal P}
\end{equation}
of the projective bundle in \eqref{e8}.

Let $T_{\varphi}\, \subset\, T {\mathbf P}_{\mathcal P}$ be the relative holomorphic tangent bundle for
the projection $\varphi$ in \eqref{e8}, so $T_{\varphi}$ is the kernel of the differential $d\varphi$
of $\varphi$. It is straightforward to check that its direct image is the
following: $$\varphi_*T_{\varphi}\,\,=\,\, J^2(TM).$$

The holomorphic connection $\nabla^{\mathcal P}$ on the fiber bundle
${\mathbf P}_{\mathcal P}$ (see \eqref{e12}) produces a projection
\begin{equation}\label{pn}
P_\nabla\,\, :\,\, T {\mathbf P}_{\mathcal P}\,\longrightarrow\, T_\varphi
\end{equation}
whose kernel is the horizontal distribution for $\nabla^{\mathcal P}$. Let
$$
d\sigma\,\, :\,\, TM \,\, \longrightarrow\, \, \sigma^* T {\mathbf P}_{\mathcal P}
$$
be the differential of the map $\sigma$ in \eqref{e9}. The holomorphic homomorphism
\begin{equation}\label{e13}
{\mathbf S}_\sigma\,\,:=\, \, (\sigma^* P_\nabla)\circ d\sigma\,\, :\,\, 
TM \,\, \longrightarrow\, \, \sigma^*T_\varphi,
\end{equation}
where $P_\nabla$ is the projection in \eqref{pn}, is the
second fundamental form of $\sigma$ for the connection $\nabla^{\mathcal P}$.

{}From the constructions of $\sigma$ and $\nabla^{\mathcal P}$ it follows immediately that
the homomorphism ${\mathbf S}_\sigma$ in \eqref{e3} is actually an isomorphism. Indeed, when
$M\,=\,{\mathbb C}{\mathbb P}^1$, then $${\mathbf P}_{\mathcal P}\,=\, {\mathbb C}{\mathbb P}^1
\times {\mathbb C}{\mathbb P}^1\, \stackrel{\varphi}{\longrightarrow}\,
{\mathbb C}{\mathbb P}^1$$ and $\nabla^{\mathcal P}$ is the trivial connection (in fact, there
is no nontrivial connection on this fiber bundle). The section $\sigma$ is the diagonal map
of ${\mathbb C}{\mathbb P}^1$. Hence in this case ${\mathbf S}_\sigma$ is an isomorphism. From this
it follows immediately that the homomorphism in \eqref{e13} is an isomorphism over coordinate charts
compatible with the projective structure $\mathcal P$ on $M$. Consequently, ${\mathbf S}_\sigma$
in \eqref{e13} is an isomorphism.

Consider the triple $({\mathbf P}_{\mathcal P},\, \nabla^{\mathcal P},\, \sigma)$ constructed
in \eqref{e8}, \eqref{e12} and \eqref{e9} respectively from the projective structure $\mathcal P$ on $M$.
The projective structure $\mathcal P$ can actually be recovered from this triple. This will be
explained below.

Let $\varphi\, :\, {\mathbf P}\, \longrightarrow\, M$ be a holomorphic fiber bundle whose typical fiber is
${\mathbb C}{\mathbb P}^1$, and let $\nabla$ be a holomorphic connection on $\mathbf P$. Let
$\sigma\,:\, M\, \longrightarrow\, {\mathbf P}$ be a holomorphic section of $\varphi$ such that the
second fundamental form
\begin{equation}\label{e14}
{\mathbf S}_\sigma\,\, :\,\, TM \,\, \longrightarrow\, \, \sigma^*T_\varphi
\end{equation}
(see \eqref{e13}) is an isomorphism.

Fix a point $x_0\,\in\, M$. Let
\begin{equation}\label{uc}
\varpi\,:\, (\widetilde{M}, \, \widetilde{x}_0)\, \longrightarrow\,
(M, \, x_0)
\end{equation}
be the corresponding universal covering. Then the flat ${\mathbb C}{\mathbb P}^1$--bundle
$(\varpi^*{\mathbf P},\, \varpi^*\nabla)$ on $\widetilde M$ is canonically identified with
$({\widetilde M}\times{\mathbf P}_{x_0},\, \nabla_0)$, where ${\mathbf P}_{x_0}$ is the fiber
of $\mathbf P$ over $x_0$ and $\nabla_0$ is the holomorphic
connection given by the trivialization of the trivial fiber bundle ${\widetilde M}
\times{\mathbf P}_{x_0} \, \longrightarrow\, \widetilde{M};$ this isomorphism
of flat ${\mathbb C}{\mathbb P}^1$--bundles
over $\widetilde M$ is given by parallel translations along paths
on $\widetilde M$ emanating from $\widetilde{x}_0$. Using the above identification of $\varpi^*{\mathbf P}$
with ${\widetilde M}\times{\mathbf P}_{x_0}$, we have the composition of maps
$$
\widetilde{M} \, \xrightarrow{\,\,\, \varpi^*\sigma\,\,}\, \varpi^*{\mathbf P} \,
\xrightarrow{\,\,\,\sim\,\,\,}\, {\widetilde M}\times{\mathbf P}_{x_0}\,\longrightarrow\, {\mathbf P}_{x_0},
$$
where the last map is the natural projection to ${\mathbf P}_{x_0}$, and the isomorphism
$\varpi^*{\mathbf P} \,\stackrel{\sim}{\longrightarrow}\,
{\widetilde M}\times{\mathbf P}_{x_0}$ is the one described above. Let
$$
\gamma\,:\, \widetilde{M} \,\longrightarrow\, {\mathbf P}_{x_0} \,=\, {\mathbb C}{\mathbb P}^1
$$
denote this composition of maps. The given condition that ${\mathbf S}_\sigma$ in \eqref{e14} is
an isomorphism implies that $\gamma$ is locally a biholomorphism.

Now construct holomorphic coordinate charts on $M$ of the form $\gamma\circ\varpi^{-1}$ (the covering
map $\varpi$ in \eqref{uc} is locally invertible). It is easy to see that all the transition functions
for these coordinate charts are M\"obius
transformations. Consequently, we obtain a projective structure on $M$.

The projective structure $\mathcal P$ is recovered back from the triple
$({\mathbf P}_{\mathcal P},\, \nabla^{\mathcal P},\, \sigma)$ (constructed
in \eqref{e8}, \eqref{e12} and \eqref{e9} respectively) in the above way.

\section{Transversal projective structures}\label{se 4}

\subsection{Foliation and projective structure}

Now let $M$ be a complex manifold of complex dimension $d+1$, with $d\, \geq\, 1$. Let
$$
{\mathcal F}\,\, \subset\,\, TM
$$
be a nonsingular holomorphic foliation of dimension $d$. Let
\begin{equation}\label{e15}
{\mathcal N}\,\, :=\,\, (TM)/{\mathcal F}
\end{equation}
be the normal bundle to $\mathcal F$; it is a holomorphic line bundle. Since the sheaf $\mathcal F$ is
closed under the operation of Lie bracket of vector fields on $M$, the Lie bracket operation
produces a holomorphic homomorphism
\begin{equation}\label{e16}
{\mathcal D}_N\,\, :\,\, {\mathcal N} \,\, \longrightarrow\,\, {\mathcal N}\otimes {\mathcal F}^*
\end{equation}
which satisfies the Leibniz identity. In other words, ${\mathcal D}_N$ is a holomorphic partial connection
on $\mathcal N$ in the direction of $\mathcal F$. This partial connection is flat, which a consequence of
the Jacobi identity for the Lie bracket of vector fields. It is classically known as the {\it Bott connection}.

A $\mathcal F$--chart on $M$ is
a pair of the form $(U,\, f)$, where $U\, \subset\, M$ is an open subset and
$$
f\, :\, U\, \longrightarrow\, {\mathbb C}{\mathbb P}^1
$$
is a holomorphic submersion, such that each connected component of any leaf of the foliation
$(U,\, {\mathcal F}\big\vert_U)$ is mapped, by $f$, to a point of ${\mathbb C}{\mathbb P}^1$. A
$\mathcal F$--atlas is a collection of $\mathcal F$--charts $\{(U_i,\, f_i)\}_{i\in I}$ on $M$ such
that $\bigcup_i U_i\,=\, M$.

A $\mathcal F$--projective structure on $M$ is defined by a $\mathcal F$--atlas
$\{(U_i,\, f_i)\}_{i\in I}$ on $M$ such that for any $i,\, j\, \in\, I$, and any nonempty connected
component $V\, \subset\, U_i\bigcap U_j$, there is a M\"obius transformation
${\mathcal T}^V_{j,i}\, \in\, {\rm Aut}({\mathbb C}{\mathbb P}^1)$ such that the following diagram
is commutative:
$$
\begin{matrix}
V & \xrightarrow{\,\,\, {\rm Id}_V\,\,} & V\\
\,\,\,\, \Big\downarrow f_i && \,\,\, \Big\downarrow f_j\\
{\mathbb C}{\mathbb P}^1& \xrightarrow{\,\,\, {\mathcal T}^V_{j,i}\,\,} & {\mathbb C}{\mathbb P}^1
\end{matrix}
$$
\cite{Sc}, \cite{LP}, \cite{IM}, \cite{Se}.

Two such $\mathcal F$--atlases $\{(U_i,\, f_i)\}_{i\in I}$ and $\{(U_i,\, f_i)\}_{i\in J}$ on $M$
define the same $\mathcal F$--projective structure if the $\mathcal F$--atlas on $M$ given
by their union $\{(U_i,\, f_i)\}_{i\in I\cup J}$ also satisfies the above condition.
A $\mathcal F$--chart $(U,\, f)$ is said to be compatible with the $\mathcal F$--projective structure
on $M$ defined by a $\mathcal F$--atlas $\{(U_i,\, f_i)\}_{i\in I}$ if the union
$\{(U_i,\, f_i)\}_{i\in I}\bigcup \{(U,\, f)\}$ defines a $\mathcal F$--projective structure on $M$.

Consider the normal bundle $\mathcal N$ in \eqref{e15}. The flat holomorphic partial connection
${\mathcal D}_N$ on $\mathcal N$ (see \eqref{e16}) induces a flat holomorphic partial connection on
$({\mathcal N}^*)^{\otimes 2}$. Let
\begin{equation}\label{e17}
H^0_{\mathcal F}(M,\, ({\mathcal N}^*)^{\otimes 2})\,\, \subset\,\, H^0(M,\, ({\mathcal N}^*)^{\otimes 2})
\end{equation}
be the subspace of flat holomorphic sections.

\begin{lemma}\label{lem1a}
If $M$ admits a $\mathcal F$--projective structure, then the space of all $\mathcal F$--projective
structures on $M$ is an affine space for the vector space $H^0_{\mathcal F}(M,\,
({\mathcal N}^*)^{\otimes 2})$ defined in \eqref{e17}.
\end{lemma}

\begin{proof}
The lemma follows immediately from the fact that the space of all
projective structures on a Riemann surface $Y$ is an affine space for $H^0(Y,\, K^{\otimes 2}_Y)$.
This will be briefly explained.

Let $P$ and $P_1$ be two $\mathcal F$--projective structures on $M$. Take a $\mathcal F$--chart
$(U,\, f)$ compatible with $P$, and also a holomorphic immersion
$$
\gamma\,:\, f(U) \, \longrightarrow\, {\mathbb C}{\mathbb P}^1
$$
such that $(U,\, \gamma\circ f)$ is a $\mathcal F$--chart compatible with $P_1$. While $f(U)$ has the
natural projective structure ${\mathcal P}$ given by its inclusion in ${\mathbb C}{\mathbb P}^1$, it
has another projective structure ${\mathcal P}_1$ given by $\gamma$ (the natural projective structure of
$\gamma(f(U))$ induces a projective structure on $f(U)$). Consequently, we get a holomorphic
quadratic differential
$$
\omega\, \in\, H^0(f(U),\, K^{\otimes 2}_{f(U)})
$$
for which ${\mathcal P}_1\,=\, {\mathcal P}+\omega$. Now it is easy to see that $f^*\omega$ is a
holomorphic section of $({\mathcal N}^*)^{\otimes 2}\big\vert_U$
which is flat with respect to the partial connection. It is also straight-forward to check that
this holomorphic section of $({\mathcal N}^*)^{\otimes 2}\big\vert_U$ is actually independent of the
choices of $f$ and $\gamma$. Consequently, we get an element of $H^0_{\mathcal F}(M,\,
({\mathcal N}^*)^{\otimes 2})$.

The converse is equally straightforward. Take $P$ and $(U,\, f)$ as above, and take any
holomorphic section $\theta$ of $({\mathcal N}^*)^{\otimes 2}\big\vert_U$
which is flat with respect to the partial connection. Then $\theta$ produces a section
$$
\theta'\, \in\, H^0(U,\, K^{\otimes 2}_U).
$$
Using $\theta'$ and the natural projective structure on $f(U)$
we get a new projective structure on $f(U)$. This new projective structure on $f(U)$ and the
map $f$ together give a new $\mathcal F$--projective structures on $U$.
\end{proof}

Unlike in the case of Riemann surfaces, there may not be any $\mathcal F$--projective structure on $M$.

\subsection{Another description of $\mathcal F$-projective structures}

As before, $M$ is a complex manifold equipped with a holomorphic foliation $\mathcal F$ of co-rank one. 
Let
\begin{equation}\label{vp}
\varpi\,\, :\,\, \mathbf{P}\,\, \longrightarrow\,\, M
\end{equation}
be a holomorphic fiber bundle
whose typical fiber is ${\mathbb C}{\mathbb P}^1$. Let
\begin{equation}\label{tv}
T_\varpi\, :=\, \text{kernel}(d\varpi)\, \subset\, T\mathbf{P}
\end{equation}
be the relative holomorphic tangent bundle for the above projection $\varpi$. A holomorphic connection
on $\mathbf{P}$ is a holomorphic homomorphism
\begin{equation}\label{dn}
\nabla\,\, :\,\, T\mathbf{P}\,\,\longrightarrow\,\, T_\varpi
\end{equation}
such that composition of maps
$$
T_\varpi\, \hookrightarrow\, T\mathbf{P}\, \xrightarrow{\,\,\, \nabla\,\,}\, T_\varpi
$$
is the identity map of $T_\varpi$. The holomorphic connection $\nabla$ is called integrable if
the distribution $\text{kernel}(\nabla) \, \subset\, T\mathbf{P}$ is integrable.

\begin{remark}\label{rem1}
The ${\mathbb C}{\mathbb P}^1$--bundle $\mathbf{P}$ on $M$ gives a holomorphic principal 
$\text{PGL}(2,{\mathbb C})$--bundle $E_{\rm PGL}$ on $M$. The fiber of $E_{\rm PGL}$ over any
point $m\, \in\, M$ is the space of all holomorphic isomorphisms from ${\mathbb C}{\mathbb P}^1$ to
the fiber $\mathbf{P}_m$. Giving a holomorphic connection on $\mathbf{P}$ is equivalent to
giving a holomorphic connection on the principal $\text{PGL}(2,{\mathbb C})$--bundle $E_{\rm PGL}$.
To see this, first note that the fiber bundle $\mathbf{P}$ is canonically identified with the
one associated to the principal $\text{PGL}(2,{\mathbb C})$--bundle $E_{\rm PGL}$ for the natural action of
$\text{Aut}({\mathbb C}{\mathbb P}^1) \,=\, \text{PGL}(2,{\mathbb C})$ on ${\mathbb C}{\mathbb P}^1$.
The identification is given by the map $E_{\rm PGL}\times {\mathbb C}{\mathbb P}^1\, \longrightarrow\,
{\mathbf P}$ that sends any $(\rho,\, z)\, \in\,E_{\rm PGL}\times {\mathbb C}{\mathbb P}^1$ to $\rho(z)$.
Therefore, a holomorphic connection on the principal $\text{PGL}(2,{\mathbb C})$--bundle $E_{\rm PGL}$
produces a holomorphic connection on the associated fiber bundle $\mathbf{P}$. Conversely, the direct image
$\varpi_* (T\mathbf{P}) \, \longrightarrow\, M$ (see \eqref{vp}) is identified with the Atiyah bundle
$\text{At}(E_{\rm PGL})$ of $E_{\rm PGL}$, and hence a holomorphic connection on $\mathbf{P}$ produces a
holomorphic splitting of the Atiyah exact sequence for $E_{\rm PGL}$. Therefore, a
holomorphic connection on $\mathbf{P}$ gives a holomorphic
connection on the principal $\text{PGL}(2,{\mathbb C})$--bundle $E_{\rm PGL}$. (See
\cite{At} for Atiyah bundle and the definition of connection using it.)$\hfill{\Box}$
\end{remark}

Let $\nabla$ be an integrable holomorphic connection on the holomorphic ${\mathbb C}{\mathbb P}^1$--bundle
$\mathbf{P}$ on $M$. Let
$$
\sigma\, :\, M\, \longrightarrow\, \mathbf{P}
$$
be a holomorphic section of the fiber bundle, so $\varpi\circ\sigma\,=\, {\rm Id}_M$. Consider the
differential
$$
d\sigma\, :\, TM\, \longrightarrow\, \sigma^*T {\mathbf P}
$$
of the map $\sigma$. We have the homomorphism
\begin{equation}\label{e22}
\widetilde{\mathbf S}_\sigma\, :\, TM \, \longrightarrow\, \sigma^*T_\varpi
\end{equation}
given by the following composition of maps:
$$
TM\,\, \xrightarrow{\,\,\, d\sigma\,\,}\,\, \sigma^*T {\mathbf P}\,\, \xrightarrow{\,\,\, \sigma^*\nabla\,\,}
\,\, \sigma^*T_\varpi
$$
(see \eqref{dn}).
If $$\widetilde{\mathbf S}_\sigma({\mathcal F})\,=\, 0,$$ where $\widetilde{\mathbf S}_\sigma$ is the
homomorphism in \eqref{e22}, then $\widetilde{\mathbf S}_\sigma$ produces a homomorphism
\begin{equation}\label{e23}
{\mathbf S}_\sigma\, :\, {\mathcal N}\, :=\, (TM)/{\mathcal F} \, \longrightarrow\, \sigma^*T_\varpi.
\end{equation}

We will call $\sigma$ a $\mathcal F$--section, of the holomorphic ${\mathbb C}{\mathbb P}^1$--bundle
$\mathbf P$ with holomorphic connection $\nabla$, if $\widetilde{\mathbf S}_\sigma({\mathcal F})\,=\, 0$.

\begin{lemma}\label{lem2a}
Giving a $\mathcal F$--projective structure on the foliated manifold $(M,\, {\mathcal F})$ is equivalent
to giving a triple $({\mathbf P},\, \nabla,\, \sigma)$, where
\begin{itemize}
\item $\varpi\, :\, \mathbf{P}\, \longrightarrow\, M$ is a holomorphic fiber bundle
whose typical fiber is ${\mathbb C}{\mathbb P}^1$,

\item $\nabla$ is an integrable holomorphic connection on $\mathbf P$, and

\item $\sigma\, :\, M\, \longrightarrow\, \mathbf{P}$ is a holomorphic 
$\mathcal F$--section of $\varpi$ such that the homomorphism ${\mathbf S}_\sigma\, :\, {\mathcal N}\,
\longrightarrow\, \sigma^*T_\varpi$ given by $\sigma$ (see \eqref{e23}) is an isomorphism.
\end{itemize}
\end{lemma}

\begin{proof}
The proof is identical to the constructions in Section \ref{se3.2}; we omit the details.
\end{proof}

\section{$\mathcal F$--projective structure on the torus}\label{se 5}

\subsection{Properties of unique $\mathcal F$--projective structure}\label{se4.1}

Consider the holomorphic foliation $\mathcal F$ on $C\times X$ in Section \ref{se2}. We will investigate the
existence of $\mathcal F$--projective structures on $C\times X$. As in \eqref{e15}, let
\begin{equation}\label{e18}
{\mathcal N}\,\, :=\,\, T(C\times X)/{\mathcal F}\,\, \longrightarrow\,\, C\times X
\end{equation}
be the normal bundle. Let
\begin{equation}\label{e18b}
q\,\, :\,\, T(C\times X)\,\, \longrightarrow\,\, {\mathcal N}
\end{equation}
be the corresponding quotient map.

\begin{proposition}\label{prop1}
There can be at most one $\mathcal F$--projective structures on $C\times X$.
\end{proposition}

\begin{proof}
In view of Lemma \ref{lem1a}, it suffices to show that
\begin{equation}\label{e19}
H^0_{\mathcal F}(C\times X,\, ({\mathcal N}^*)^{\otimes 2})\,\, =\,\,0,
\end{equation}
where $\mathcal N$ is the line bundle in \eqref{e18}.

To prove \eqref{e19}, fix a point $z\, \in\, X$. Consider the curve
\begin{equation}\label{e26}
C^z\, \,:=\,\, C\times\{z\}\,\, \subset\,\, C\times X.
\end{equation}
Let
\begin{equation}\label{e27}
{\mathcal N}^z\, \longrightarrow\, C
\end{equation}
be the restriction of ${\mathcal N}$ to the curve $C^z$ in \eqref{e26}. We will prove the following:
\begin{equation}\label{e20}
\text{degree}({\mathcal N}^z)\,\,=\,\, d
\end{equation}
(see \eqref{e1} for $d$).

Consider the composition of homomorphisms
$$
TC^z \, \hookrightarrow\, T(C\times X)\vert_{C^z} \, \xrightarrow{\,\,\, q\big\vert_{C^z}\,\,}\,
{\mathcal N}^z,
$$
where $q$ is the map in \eqref{e18b} and ${\mathcal N}^z$ is defined in \eqref{e27}; let
\begin{equation}\label{e20b}
\Phi^z\, :\, TC^z \, \longrightarrow\, {\mathcal N}^z
\end{equation}
be this composition of map. From the construction of $\mathcal F$ it follows immediately that
\begin{equation}\label{e20c}
\text{Div}(\Phi^z)\,=\, \text{Div}(\omega)\,=\,
\underline{y}\times \{z\} \, \subset\, C^z\, \subset\, C\times X
\end{equation}
(see \eqref{en2} for $\omega$). Since $TC^z$ is a holomorphically trivial
line bundle, from \eqref{e20c} it follows that \eqref{e20} holds. Now from
\eqref{e20} it follows immediately that $H^0(C^z,\, (({\mathcal N}^z)^{\otimes 2})^*)\,=\,0$, and this
implies that \eqref{e19} holds. As noted before, \eqref{e19} completes the proof of the proposition.
\end{proof}

Assume that there is a $\mathcal F$--projective structure on $C\times
X$. In view of Proposition \ref{prop1}, this implies that there is exactly one $\mathcal F$--projective
structure on $C\times X$. Let
\begin{equation}\label{e21}
P_0
\end{equation}
denote this $\mathcal F$--projective structure on $C\times X$. Let
\begin{equation}\label{e24}
({\mathbf P},\, \nabla,\, \sigma)
\end{equation}
be the triple corresponding to $P_0$ in \eqref{e21} (see Lemma \ref{lem2a}). Let
\begin{equation}\label{e24b}
\varpi\,\, :\,\, \mathbf{P}\,\, \longrightarrow\,\,C\times X
\end{equation}
be the natural projection.

Fix a point $x_0\, \in\, X$. Let
\begin{equation}\label{eg}
G\,\, \subset\,\, \text{Aut}(X)
\end{equation}
be the unique maximal connected subgroup of the group of all holomorphic automorphisms of $X$.
So $G$ is identified with $X$ by sending any $g\, \in\, G$ to $g(x_0)\, \in\, X$.
The natural action of $G$ on $X$ and the trivial action of $G$ on $C$ together produce an action of
$G$ on $C\times X$. The foliation $\mathcal F$ is evidently
preserved by this action of $G$ on $C\times X$ (this follows immediately from the fact that
$\beta$ in \eqref{e2} is preserved by the action of $G$ on $X$). Since
the $\mathcal F$--projective structure in \eqref{e21} is unique, the triple $({\mathbf P},\, \nabla,\, \sigma)$
in \eqref{e24} is preserved by the action of $G$.

That the pair $({\mathbf P},\, \sigma)$ is $G$--equivariant means the following:
There is a holomorphic fiber bundle
\begin{equation}\label{e28}
\rho\, :\, \mathbb{P} \, \longrightarrow\, C
\end{equation}
whose typical fiber is ${\mathbb C}{\mathbb P}^1$, and there is a holomorphic section
\begin{equation}\label{e28b}
\eta\,\, :\,\, C\,\, \longrightarrow\,\, \mathbb{P}
\end{equation}
of $\rho$, such that
\begin{equation}\label{e29}
(\phi_C^*\mathbb{P},\, \phi^*_C\eta)\,=\, ({\mathbf P},\, \sigma),
\end{equation}
where $\phi_C$ is the projection in \eqref{e3}. The $G$--equivariance
condition of $\nabla$ is a bit more involved; we first need to set up notation for describing it.

Let $T_\rho\, \subset\, T\mathbb{P}$ be the relative holomorphic tangent bundle for the projection
$\rho$ in \eqref{e28}, so $T_\rho$ is the kernel of the differential
\begin{equation}\label{e30}
d\rho\, :\, T\mathbb{P} \, \longrightarrow\, \rho^*TC
\end{equation}
of $\rho$. The direct image
\begin{equation}\label{e31}
\rho_*T_\rho\,\, \longrightarrow\,\, C
\end{equation}
is a rank three holomorphic vector bundle whose fibers are Lie algebras isomorphic
to $sl(2,{\mathbb C})$; the Lie algebra structure is given by the Lie bracket of vertical vector
fields for $\rho$. Recall from Remark \ref{rem1} that $\mathbb{P}$ gives a holomorphic principal
$\text{PGL}(2,{\mathbb C})$--bundle on $C$. The Lie algebra bundle $\rho_*T_\rho$ in \eqref{e31}
is in fact the adjoint vector bundle of this principal $\text{PGL}(2,{\mathbb C})$--bundle.

Take any
\begin{equation}\label{e32}
\theta\,\, \in\, \, H^0(C,\, (\rho_*T_\rho)\otimes_{\mathbb C} H^0(X,\, K_X))\,\,=\,\,
H^0(C,\, (\rho_*T_\rho))\otimes_{\mathbb C} H^0(X,\, K_X).
\end{equation}
Using $(d\phi_X)^*\, :\, \phi^*K_X\, \longrightarrow\, \Omega^1_{C\times X}\,=\, (\phi^*_CK_C) \oplus
(\phi^*_X K_X)$ (see \eqref{e3}) we have
$$
H^0(X,\, K_X) \, \hookrightarrow\, H^0(C\times X,\, \Omega^1_{C\times X}).
$$
Using this inclusion map together with the isomorphism $\phi_C^*\mathbb{P}\, \stackrel{\sim}{\longrightarrow}\,
{\mathbf P}$ in \eqref{e29}, the section $\phi^*_C \theta$ (see \eqref{e32}) produces a section
\begin{equation}\label{e33}
\widetilde{\theta}\, \,\in\,\, H^0(C\times X,\, \varpi_* T_\varpi)\otimes H^0(C\times X,\,
\Omega^1_{C\times X})\,=\, H^0(C\times X,\, \varpi_*T_\varpi\otimes \Omega^1_{C\times X}),
\end{equation}
where $T_\varpi\, \subset\, T{\mathbf P}$, as in \eqref{tv}, is the kernel of $d\varpi$ for
$\varpi$ in \eqref{e24b}; note that
$\phi^*_C \rho_* T_\rho\,=\, \varpi_*T_\varpi$, which follows from the isomorphism
$\phi_C^*\mathbb{P}\, \stackrel{\sim}{\longrightarrow}\, {\mathbf P}$ in \eqref{e29}.

Recall that $\varpi_* T_\varpi\, \longrightarrow\, C\times X$ is the adjoint vector bundle of the holomorphic 
principal $\text{PSL}(2,{\mathbb C})$--bundle on $C\times X$ given by the ${\mathbb C}{\mathbb P}^1$--fiber 
bundle $\mathbf P$ (see Remark \ref{rem1}). Therefore, if $\nabla^1$ is a holomorphic connection on the fiber 
bundle ${\mathbf P}$, then $\nabla^1+\widetilde{\theta}$ is also a holomorphic connection on the fiber bundle 
${\mathbf P}$, where $\widetilde{\theta}$ is the section in \eqref{e33}.

The $G$--equivariance (see \eqref{eg}) condition of the connection $\nabla$ (see \eqref{e24}) says that
there is a pair
\begin{equation}\label{e35}
(\nabla',\, \theta),
\end{equation}
where
\begin{itemize}
\item $\nabla'$ is a holomorphic connection on the fiber bundle $\rho\, :\,\mathbb{P}\, \longrightarrow\, C$
(see \eqref{e28}), and

\item $\theta \, \in\, H^0(C,\, (\rho_*T_\rho))\otimes_{\mathbb C} H^0(X,\, K_X)$ (see \eqref{e32}),
\end{itemize}
such that
\begin{equation}\label{e34}
\nabla\,\,=\,\, \phi^*_C\nabla' + \widetilde{\theta},
\end{equation}
where $\widetilde{\theta}$ is constructed
from $\theta$ as in \eqref{e33}; using the isomorphism $\phi_C^*\mathbb{P}\, \stackrel{\sim}{\longrightarrow}
\, {\mathbf P}$ in \eqref{e29} the connection $\phi^*_C\nabla'$ on $\phi_C^*\mathbb{P}$ is considered ---
in \eqref{e34} --- as a holomorphic connection on the fiber bundle $\mathbf P$, ad hence
$\phi^*_C\nabla' + \widetilde{\theta}$ is also a holomorphic connection on the fiber bundle $\mathbf P$.

We note that $\nabla'$ and $\theta$ are uniquely determined by $\nabla$ and \eqref{e34}.

\begin{remark}\label{rem2}
The triple $({\mathbf P},\, \nabla,\, \sigma)$ in \eqref{e24} actually determines the foliation
$\mathcal F$ uniquely. To see this, we first observe that for any $x\, \in\, X$ and $y_i$,
$i\, \in\, \{1,\, \cdots, \, d\}$ (see \eqref{e1}), the line
$${\mathcal F}_{(y_i,x)}\,\, \subset\,\, T_{(y_i,x)}(C\times X)$$ coincides with
$T_{y_i} C\, \subset\, T_{(y_i,x)}(C\times X)$. Consequently, the homomorphism
$$
\widetilde{\mathbf S}_\sigma\,:\, T(C\times X)\, \longrightarrow\,\sigma^* T_\varpi
$$
(see \eqref{e22} and \eqref{e23}) vanishes on the subspace $T_{y_i} C\, \subset\, T_{(y_i,x)}(C\times X)$;
recall that the third condition in Lemma \ref{lem2a} says that
$\sigma$ is a holomorphic $\mathcal F$--section of $\varpi$ which implies that
$\widetilde{\mathbf S}_\sigma (T_{y_i} C)\,=\, 0$.

For any point $y\, \in\, C\setminus \{y_1,\, \cdots,\, y_d\}$, the line
${\mathcal F}_{(y,x)}\,\, \subset\,\, T_{(y,x)}(C\times X)$ does not coincide with
$T_{y} C\, \subset\, T_{(y_i,x)}(C\times X)$. Since ${\mathbf S}_\sigma$ in
the third condition in Lemma \ref{lem2a} is an isomorphism, this implies that
$\widetilde{\mathbf S}_\sigma (T_{y} C)\,\not=\, 0$.

Therefore, the subset $\{y_1,\, \cdots,\, y_d\}\, \subset\, Y$ is determined by
the triple $({\mathbf P},\, \nabla,\, \sigma)$.

Take a point $c\, \in\, C$ and restrict the triple $({\mathbf P},\, \nabla,\, \sigma)$
over $\{c\}\times X \, \subset\, C\times X$. Note that the line subbundle
${\mathcal F}\big\vert_{\{c\}\times X}\, \subset\, T(C\times X)\big\vert_{\{c\}\times X}$
coincides with $$TX\,=\, T(\{c\}\times X)\, \subset\, T(C\times X)\big\vert_{\{c\}\times X}$$
if and only of $c\, \in\, \{x_1,\, \cdots,\, x_d\}$ (see \eqref{e1}). Consequently,
$\sigma\big\vert_{\{c\}\times X}$ is a flat section of $({\mathbf P},\, \nabla)\big\vert_{\{c\}\times X}$
if and only of $c\, \in\, \{x_1,\, \cdots,\, x_d\}$.

{}From these observations we conclude that the triple $({\mathbf P},\, \nabla,\, \sigma)$ in \eqref{e24}
determines the foliation $\mathcal F$ uniquely.$\hfill{\Box}$
\end{remark}

Actually, it can be shown that every holomorphic connection on $\mathbf P$ is of the form
given in \eqref{e34}. In other words, every holomorphic connection on $\phi_C^*\mathbb{P}$ is automatically
$G$--equivariant. Even though it is not needed here, this statement is proved below.

\begin{lemma}\label{lem3}
Let ${\mathbb P}\, \longrightarrow\, C$ be a holomorphic fiber bundle whose typical fiber
is ${\mathbb C}{\mathbb P}^1$. Then every holomorphic connection on the fiber bundle
$\phi_C^*\mathbb{P}\, \longrightarrow\, C\times X$ is of the form given in \eqref{e34}.
\end{lemma}

\begin{proof}
Let $\widehat{\nabla}$ be a holomorphic connection on $\phi_C^*\mathbb{P}$. For any $x\, \in\, X$,
let $\widehat{\nabla}^x$ be the holomorphic connection on $(\phi_C^*\mathbb{P})\big\vert_{C\times\{x\}}
\,=\, {\mathbb P}$ over $C\times \{x\}\,=\, C$ obtained by restricting $\widehat{\nabla}$ to
$C\times\{x\}$. This produces a holomorphic map from $X$ to the space of holomorphic connections on
${\mathbb P}$. Since the space of holomorphic connections on ${\mathbb P}$ is an affine space for
a complex vector space, this map must be a constant one.

Next take any $c\, \in\, C$. The restriction of $\phi_C^*\mathbb{P}$ to $\{c\}\times X$ is the trivial
holomorphic bundle $X\times {\mathbb P}_c\, \longrightarrow\, X$, where ${\mathbb P}_c$ is the fiber
of $\mathbb P$ over $c$. Let $\nabla_0$ denote the trivial connection on this trivial fiber bundle
$X\times {\mathbb P}_c\, \longrightarrow\, X$. Let $\widehat{\nabla}_c$ denote the restriction of
$\widehat\nabla$ to
$$
(\phi_C^*\mathbb{P})\big\vert_{\{c\}\times X}\,=\, X\times {\mathbb P}_c\, \longrightarrow\, X\,=\,
\{c\}\times X.
$$
So we have
\begin{equation}\label{dc}
\widehat{\nabla}_c - \nabla_0 \, \in\, H^0(X,\, V_c\otimes_{\mathbb C} K_X),
\end{equation}
where $V_c$ is the fiber, over $c\, \in\, C$, of the adjoint vector bundle for the
principal $\text{PGL}(2,{\mathbb C})$ on $C$ given by $\mathbb P$ (see Remark \ref{rem1}). But any
holomorphic section of a trivial vector bundle on $X$ is a constant section; in particular,
all elements of $H^0(X,\, V_c\otimes_{\mathbb C} K_X)$ are constant sections (meaning the sections
are $G$--invariant, where $G$ is the group in \eqref{eg}). Hence
$\widehat{\nabla}_c - \nabla_0$ in \eqref{dc} is given by a constant section. From these observations
it follows that every holomorphic connection on the fiber bundle
$\phi_C^*\mathbb{P}\, \longrightarrow\, C\times X$ is of the form given in \eqref{e34}.
\end{proof}

In the proof of Lemma \ref{lem3} the condition that $\text{genus}(C)\,=\, 1$ is not used, but
the condition that $\text{genus}(X)\, \leq\, 1$ is crucially used.

\subsection{Nonexistence of $\mathcal F$--projective structure}

Consider $\text{Sym}^d_0(C)$ defined in Section \ref{se2}. Let
$$
\widehat{\mathcal U}\,\, \subset\,\, \text{Sym}^d_0(C)\times \text{Sym}^d_0(C)
$$
be the Zariski open subset consisting of all $\{x_1,\, \cdots\, x_d\},\,
\{y_1,\, \cdots\, y_d\}\,\in\, \text{Sym}^d_0(C)$ such that
$$
\{x_1,\, \cdots,\, x_d\}\cap\{y_1,\, \cdots,\, y_d\}\, =\, \emptyset.
$$
Let
\begin{equation}\label{eu}
{\mathcal U}\,\, \subset\,\, \widehat{\mathcal U}
\end{equation}
be the Zariski closed subvariety consisting of all $(\{x_1,\, \cdots\, x_d\},\,
\{y_1,\, \cdots\, y_d\})\,\in\, \widehat{\mathcal U}$ such that the 
two elements $\sum_{i=1}^d x_i$ and $\sum_{i=1}^d y_i$ of $C$ coincide (recall that
this condition is independent of the choice of the identity element of $C$). Denote
the multiplicative group ${\mathbb C} \setminus\{0\}$ by ${\mathbb C}^\star$. Let
$$
\widehat{\mathcal E}\,\, \longrightarrow\,\, {\mathcal U}
$$
be the holomorphic principal ${\mathbb C}^\star\times {\mathbb C}^\star$--bundle
whose fiber over any $(\{x_1,\, \cdots\, x_d\},\,
\{y_1,\, \cdots\, y_d\})\,\in\, \widehat{\mathcal U}$ is a pair $(\widehat{\omega},\, \widehat{\beta})$,
where $\widehat{\beta}\, \in\, H^0(X,\, K_X)\setminus \{0\}$ and $\widehat{\omega}$ is a meromorphic
$1$--form on $C$ whose polar divisor is exactly $\{x_1,\, \cdots\, x_d\}$ and the zero divisor is exactly
$\{y_1,\, \cdots\, y_d\}$; see Lemma \ref{lem2}. Consider the diagonal
subgroup
$$
{\mathbb C}^\star\,\, \hookrightarrow\,\, {\mathbb C}^\star\times {\mathbb C}^\star,\,\ \
c\, \longmapsto\, (c,\, c),
$$
and let
\begin{equation}\label{pb}
{\mathcal E}\,\, :=\, \, \widehat{\mathcal E}/{\mathbb C}^\star\,\, \longrightarrow\,\, {\mathcal U}
\end{equation}
be the quotient by this subgroup. So $\mathcal E$ is a holomorphic principal ${\mathbb C}^\star$--bundle
on $\mathcal U$.

Note that the construction in Section \ref{se2} produces a holomorphic foliation ${\mathcal F}_z$
of rank one on $C\times X$ for every element of $z\, \in\, {\mathcal E}$, where $\mathcal E$
is constructed in \eqref{pb}.

\begin{theorem}\label{thm1}
Assume that $d\, \geq\, 8$. There is a nonempty Zariski open subset 
$$
\mathbf{U}\,\, \subset\,\, {\mathcal E}
$$
such that for any $z\, \in\, \mathbf{U}$, the corresponding holomorphic foliation ${\mathcal F}_z$
of rank one on $C\times X$ satisfy the following condition: There is no ${\mathcal F}_z$--projective
structure on $C\times X$.
\end{theorem}

\begin{proof}
First note that
\begin{equation}\label{e25}
\dim {\mathcal E}\,\,=\,\, \dim {\mathcal U} +1 \,\,=\,\, 2d-1 +1\,\,=\,\, 2d.
\end{equation}

Consider a triple $({\mathbb P},\, \eta,\, \nabla')$ on $C$,
where
\begin{enumerate}
\item $\rho\, :\, {\mathbb P}\,\longrightarrow\, C$ is a holomorphic fiber bundle whose typical
fiber is ${\mathbb C}{\mathbb P}^1$,

\item $\eta:\, C\, \longrightarrow \, {\mathbb P}$ is a holomorphic section of $\rho$, and

\item $\nabla'$ is a holomorphic connection on $\mathbb P$.
\end{enumerate}
Let $T_\rho\, \subset\, T{\mathbb P}$ be the holomorphic line subbundle given by the kernel of
the differential of $\rho$; so $T_\rho$ is the vertical tangent bundle for the projection $\rho$.
Let $$\psi\, :\, T{\mathbb P}\, \longrightarrow\, T_\rho$$ be the projection given by the
connection $\nabla'$, so the kernel of $\psi$ is the horizontal subbundle of $T{\mathbb P}$ for
$\nabla'$. Consider the differential
$$
d\eta\,\, :\,\, TC\,\, \longrightarrow\,\, \eta^*T{\mathbb P}
$$
of the section $\eta$. Let
\begin{equation}\label{e36}
(\eta^*\psi) \circ (d\eta)\, \, :\,\, TC\,\, \longrightarrow\,\, \eta^*T_\rho
\end{equation}
be the composition of homomorphisms.

Let us now estimate the dimension of all quadruples $({\mathbb P},\, \eta,\, \nabla',\, \theta)$ on $C$, where
\begin{enumerate}
\item $\rho\, :\, {\mathbb P}\,\longrightarrow\, C$ is a holomorphic fiber bundle whose typical
fiber is ${\mathbb C}{\mathbb P}^1$,

\item $\eta:\, C\, \longrightarrow \, {\mathbb P}$ is a holomorphic section of $\rho$,

\item $\nabla'$ is a holomorphic connection on $\mathbb P$ (see \eqref{e29}), and

\item $\theta\, \in\, H^0(C,\, (\rho_*T_\rho)\otimes_{\mathbb C} H^0(X,\, K_X))\, =\,
H^0(C,\, (\rho_*T_\rho))\otimes_{\mathbb C} H^0(X,\, K_X)$ (see \eqref{e32}),
\end{enumerate}
such that the homomorphism $(\eta^*\psi) \circ (d\eta)$ in
\eqref{e36} vanishes on exactly $d$ distinct points of $C$.

The holomorphic ${\mathbb C}{\mathbb P}^1$--fiber bundle $\mathbb P$ admits a holomorphic connection,
using which it can be shown that the dimension of the space of all such holomorphic ${\mathbb C}{\mathbb
P}^1$--fiber bundles on $C$ is $1$. To prove this first note that the fact that
$\mathbb P$ admits a holomorphic connection implies that the holomorphic principal
$\text{PGL}(2,{\mathbb C})$--bundle on $C$ corresponding to $\mathbb P$ is semistable \cite[p.~41,
Theorem 4.1]{BG}. Consequently, $\mathbb P$ is one of the following:
\begin{enumerate}
\item projectivization ${\mathbb P}({\mathbb O}_C\oplus{\mathcal L})$, where $\mathcal L$ is a holomorphic
line bundle on $C$ of degree zero;

\item projectivization of the unique nontrivial extension of ${\mathcal O}_C$ by ${\mathcal O}_C$;

\item projectivization of the unique nontrivial extension of $L$ by ${\mathcal O}_C$, where $L$ is any
holomorphic line bundle on $C$ of degree $1$.
\end{enumerate}
Consequently, the dimension of the space of all holomorphic ${\mathbb C}{\mathbb P}^1$--fiber bundles
on $C$ admitting a holomorphic connection is $1$.

Let $\mathbb V$ denote the adjoint vector bundle of the principal $\text{PGL}(2,{\mathbb C})$--bundle
on $C$ corresponding to $\mathbb P$ (see Remark \ref{rem1}). So $\mathbb V$ is a semistable vector
bundle of rank three and degree zero; the connection $\nabla'$ produces a holomorphic
connection on $\mathbb V$, and any holomorphic vector bundle on $C$ admitting a holomorphic connection
is semistable \cite[p.~41, Theorem 4.1]{BG}. Therefore, we have
$$
\dim H^0(C, \, {\mathbb V}\otimes K_C)\,=\, \dim H^0(C, \, {\mathbb V})\,\leq \, 
{\rm rank}({\mathbb V})\,=\, 3.
$$

We have $\rho_*T_\rho\,=\, {\mathbb V}$, so
$$
\dim H^0(C,\, (\rho_*T_\rho)\otimes_{\mathbb C} H^0(X,\, K_X))\, \leq \, 3.
$$

The tangent space, at $\eta$, of the space of holomorphic sections of $\mathbb P$ is
$H^0(C,\, \eta^* T_\rho)$. Since the homomorphism $(\eta^*\psi) \circ (d\eta)$ in
\eqref{e36} vanishes on exactly $d$ distinct points of $C$, we have
${\rm degree}(\eta^* T_\rho)\,=\, d$. This implies that
$$
\dim H^0(C,\, \eta^* T_\rho)\,\,=\, \, d
$$
as $d\, \geq\,1$.

Consequently, the dimension of the space of all 
quadruples $({\mathbb P},\, \eta,\, \nabla',\, \theta)$ on $C$ of the above type is bounded above by
$$
1+3+3+d \,\,=\, d+7.
$$
We have
\begin{equation}\label{ef}
d+7\, <\, 2d\,=\, \dim {\mathcal E}
\end{equation}
(see \eqref{e25}) because $d\, \geq\, 8$.

Consider the holomorphic foliations $\mathcal F$ on $C\times X$, given by a meromorphic form on
$C$ with a reduced pole and zero of order $d$, such that there is a ${\mathcal F}$--projective
structure on $C\times X$. In Section \ref{se4.1} we saw that there is a natural injective map from this 
space to the space of quadruples $({\mathbb P},\, \eta,\, \nabla',\, \theta)$ on $C$
of above type; the injectivity of the map follows from Remark \ref{rem2}. Therefore, the proof
is completed using \eqref{ef}.
\end{proof}

\section*{Acknowledgements}

The authors are grateful to Frank Loray and Jorge Pereira for helpful discussions. The first 
author is partially supported by a J. C. Bose Fellowship (JBR/2023/000003).

%%%%%%%%%%%%%%%%%%%%%%%%%%%%%%%%%%%%%%%%%%%%%%%%%%%%%%%%%%%%%%%%%%%%%%%%%%%%%%%%%%%%%%%%%%%

\end{document}